\documentclass{article}
\usepackage{amsmath,amsthm,amssymb,cite}
\allowdisplaybreaks[3]
\numberwithin{equation}{section}
\newtheorem{theorem}{Theorem}[section]

\newtheorem{lemma}[theorem]{Lemma}
\newtheorem{proposition}[theorem]{Proposition}
\theoremstyle{definition}

\theoremstyle{remark}
\newtheorem{remark}[theorem]{Remark}

\newcommand{\N}{\mathbb{N}}

\newcommand{\R}{\mathbb{R}}
\newcommand{\X}{{\R^d}}
\newcommand{\B}{\mathcal{B}}

\newcommand{\K}{\mathcal{K}}

\newcommand{\ga}{\gamma}
\newcommand{\Ga}{\Gamma}
\newcommand{\la}{\lambda}
\newcommand{\La}{\Lambda}
\newcommand{\eps}{\varepsilon}

\newcommand{\est}{\mathrm{est}}
\newcommand{\fec}{\mathrm{fec}}
\newcommand{\Le}{L_\est}
\newcommand{\Lf}{L_\fec}
\renewcommand{\L}{\mathcal{L}}
\newcommand{\LC}{\mathcal{L}_C}

\newcommand{\hLe}{\hat{L}_\est}
\newcommand{\hLf}{\hat{L}_\fec}

\newcommand{\D}{\mathcal{D}}
\newcommand{\Dom}{\mathrm{Dom}}

\newcommand{\e}{{(\eps)}}

\begin{document}

\title{Establishment and Fecundity in~Spatial~Ecological~Models: Statistical~Approach and~Kinetic Equations}

\author{Dmitri Finkelshtein\thanks{Institute of Mathematics,
         National Academy of Sciences of Ukraine,
         01601 Kiev-4, Ukraine, e-mail:fdl@imath.kiev.ua} \and 
         Yuri Kondratiev\thanks{Fakult\"at f\"ur Mathematik,
         Universit\"at Bielefeld,
         Postfach 110 131, 33501 Bielefeld, Germany,
         e-mail: kondrat@math.uni-bielefeld.de} \and
         Oleksandr Kutoviy\thanks{Fakult\"at f\"ur Mathematik,
         Universit\"at Bielefeld,
         Postfach 110 131, 33501 Bielefeld, Germany,
         e-mail: kutoviy@math.uni-bielefeld.de}}

\maketitle

\begin{abstract}
We consider spatial population dynamics given by Markov birth-and-death process with constant mortality and birth influenced by establishment or fecundity mechanisms. The~independent and density dependent dispersion of spreading are studied. On the base of general methods of \cite{FKK2011a}, we construct the state evolution of considered microscopic ecological systems. We analyze mesoscopic limit for stochastic dynamics under consideration. The corresponding Vlasov-type non-linear kinetic equations are derived and studied.
\end{abstract}



\section{Introduction}

Complex systems theory is a quickly growing interdisciplinary area with a very broad
spectrum of motivations and applications. One may characterize complex systems by such
properties as diversity and individuality of components, localized interactions among
components, and the outcomes of interactions used for replication or enhancement of
components. In the study of these systems, proper language and techniques are delivered
by
the interacting particle models which form a rich and powerful direction in modern
stochastic and infinite dimensional analysis. Interacting particle systems are widely
used
as models in condensed matter physics, chemical kinetics, population biology, ecology,
sociology, and economics.

Mathematical realizations of such models may be considered as a dynamics of points in
proper state spaces. In some applications the possible locations for the points of
system
are structured, e.g., if we consider dynamics on graphs, or, in particular, on lattices.
Another class of models can be characterized by the free positions of points in
continuum,
e.g., in Euclidean space $\X$. As it was shown originally in statistical physics, many
empirical effects, such as phase transition, are impossible in systems with finite
number
of points. Therefore, systems with infinite points can be considered as mathematical
approximation for realistic systems with huge \emph{but finite} number of elements.
Among
all infinite systems we will study locally finite ones. Namely, the configuration space
over space $\X$ consists of all locally finite subsets (configurations) of $\X$
\begin{equation}  \label{confspace}
\Ga =\Ga\bigl(\X\bigr) :=\Bigl\{ \ga \subset \X \Bigm| |\ga
_\La
|<\infty, \ \mathrm{for \ all } \ \La \in {\B}_{\mathrm{b}}
(\X)\Bigr\}.
\end{equation}
Here $\ga_\La:=\ga\cap\La$, the symbol $|\cdot|$ stands for the cardinality of a set,
and
${\B}_{\mathrm{b}} ({\X})$ denotes the class of all bounded Borel sets in $\X$. Each
configuration may be identified with a Radon measure on $\X$ by the relation
$\ga(\La)=|\ga_\La|$. As a result, $\Ga$ can be equipped with the vague topology and the
corresponding Borel $\sigma$-algebra.

Depending on application the points of system may be interpreted as molecules in
physics,
plants in ecology, animals in biology, infected people in medicine, companies in
economics, market agents in finance, and so on. It is supposed that points of a system
evolve in time interacting with each other. In the present paper we focus our attention
to
the dynamics with birth and death mechanisms.

The spatial birth-and-death dynamics describe an evolution of configurations in $\R^d$,
in
which points of configurations (particles, individuals, elements) randomly appear (born)
and disappear (die) in the space.
Heuristically, the
corresponding Markov generator has the following form:
\begin{equation} \label{BP-gen}
(LF)(\ga)= \sum_{x\in\ga}d(x,\ga\setminus
x)D_{x}^{-}F(\ga)+\int_{\R^{d}}b(x,\ga)D_{x}^{+}F(\ga)dx,
\end{equation}
where for $F:\Ga\to\R$, $x\notin\ga$
\begin{equation}
D_{x}^{-}F(\ga)=F(\ga\setminus x)-F(\ga), \qquad
D_{x}^{+}F(\ga)=F(\ga\cup x)-F(\ga).\label{elementarydif}
\end{equation}
Here functions $d$ and $b$ describe rates of death and birth correspondingly (for details
see, e.g., \cite{FKK2011a}).

In the present paper we apply the results of \cite{FKK2011a} to study the question about
the existence of the evolution corresponding to \eqref{BP-gen} for a particular choice of
the functions $d$ and $b$. This question can be answered once we will be able to
construct a semigroup associated with $L$ in a proper functional space. This semigroup
determines the solution
to the Kolmogorov equation, which formally (only in the sense of action of operator) has
the following form:
\[
\frac{dF_t}{dt}=LF_t,\qquad
F_t\bigm|_{t=0}=F_0.
\]
To show directly that $L$ is a~generator of a~semigroup in some reasonable functional
spaces on $\Ga$ seems to be difficult problem. This difficulty is hidden in the complex
structure of non-linear infinite dimensional space $\Ga$.
However, in various applications the corresponding evolution of states (measures on
configuration space) helps already to understand the behavior of the process and makes
possible to predict the equilibrium states of our system. In fact, properties of such an
evolution itself are very important for application. The evolution of states is
heuristically given as a~solution to the dual Kolmogorov equation (Fokker--Planck
equation):
\begin{equation}
\frac{d\mu_t}{dt}=L^*\mu_t, \qquad
\mu_t\bigm|_{t=0}=\mu_0\label{FokkerPlanck},
\end{equation}
where $L^*$ is an adjoint operator to $L$ defined on some space of measures on $\Ga$,
provided, of course, that it exists.

Technically, we will study solutions of \eqref{FokkerPlanck} in terms of correlations
functions, $k_t^{(n)},\,n\geq 0$ which are symmetric functions on $(\X)^n$ and related
to
a density of distribution for each $n$ points of our system (rigorous definition will be
given in  Section 2).

Among all birth-and-death processes we will consider
only those in which new particles appear from existing ones. These processes correspond
to the models of the spatial ecology. In the recent paper \cite{FKK2009}, we studied
Bolker--Dieckmann--Law--Pacala ecological
model, which corresponds to the following mechanism of evolution. Each existing
individual
can give birth to the new one independently of all other individuals of the system. It
may
also die influenced by the global regulation (mortality) again independently of all
other
members of the population or it dies because of the interaction with the rest of the
population (local regulation). The latter mechanism may be described as a~competition
(e.g., for resources) between individuals in the population. Heuristically, the
corresponding Markov generator has the form \eqref{BP-gen} with
\begin{eqnarray}
d(x,\ga)&=&m+\varkappa^{-}\sum_{y\in\ga}a^{-}(x-y), \label{BP-deathrate}\\
b(x,\ga)&=&\varkappa^{+}\sum_{y\in\ga}a^{+}(x-y),\label{BP-birthrate}
\end{eqnarray}
Here $a^{+}, a^{-}$ are probability densities, and constants $m, \varkappa^{+},
\varkappa^{-}\geq 0$. In population ecology, the constant $m$ is called mortality and
the
functions $a^+, a^-$ are known as dispersion and competition kernel, respectively.

By \cite{FKK2009}, if $m=\varkappa^-=0$ (free growth model) then the first correlation
function (density of the system) grows exponentially in time. To suppress this growth we
may consider the case $m>\varkappa^-=0$ (contact model, see also \cite{KS2006,KKP2008}). Then for $m\geq \varkappa^+$ we obtain globally bounded density (even
decaying in time for $m>\varkappa^+$). Nevertheless, locally the system will show
clustering. Namely, $k_t^{(n)}\sim n!$ on a~small regions for $t\geq0$ (see
\cite{FKK2009}
for details). The main result of \cite{FKK2009} may be informally stated in the following
way: if the
mortality $m$ and the competition kernel $\varkappa^-a^-$ are large enough, then the
dynamics of correlation functions associated with the pre-generator \eqref{BP-gen}
preserves \mbox{(sub-)Poissonian} bound for correlation functions for all times, i.e.,
$k_t^{(n)}\leq C^n$, $C>0$, $n\geq1$.

In the present article we introduce new mechanisms of local
regulation in the corresponding system, alternatively to \eqref{BP-deathrate}. Namely,
we set $\varkappa^-=0$ in \eqref{BP-deathrate} and consider two different modifications
of \eqref{BP-birthrate}. The first one
includes the influence of the whole system on the reproduction (fertility, fecundity) of
each single individual. The second modification of \eqref{BP-birthrate} contains
a~mechanism which shows establishment of each individual in the system. The precise
descriptions are given in the next section. Such models have been actively studied in modern ecological literature, see e.g. \cite{BP1999} and references therein. Here, for the first time, we present a rigorous mathematical description for these evolutions.

This article is organized in the following way. In Section 2, we describe the model
rigorously providing the proper spaces for the corresponding functional evolutions. In
Section~3 we apply general results about birth-and-death dynamics on configuration
spaces
obtained in \cite{FKK2011a}. Informally, the main results state that if mortality $m$ is
big enough and negative influence of establishment or fecundity
is dominated by dispersion then the corresponding evolution exist. In Section~4, we
study the mesoscopic description of our model in terms of Vlasov scaling.

It should be noted also, that the Vlasov-type scalings for some
Markov processes on finite configuration spaces were considered in
\cite{Bel1988,Bel1989,BK2003,BM1977,BMT1981}. Note that the
corresponding limiting hierarchy was obtained at the heuristic
level. In the present paper, we prove a weak convergence to the
limiting hierarchy in the case of infinite continuous systems for
bounded but non-integrable densities.

It is worth pointing out that the necessity of a big mortality is a result of
perturbation
theory for linear operators which gives the existence of the corresponding dynamics for
the infinite time interval. However, with the help of another technique considered in
\cite{BKKK2011}, \cite{FKKoz2011}, we are able to show the existence of the dynamics
with
any mortality but only on finite interval of time. This result will be presented in the
forthcoming paper.

\section{Description of model}
We recall that the configuration space $\Ga$ is given by \eqref{confspace}. It is
equipped
with the vague topology, i.e., the weakest topology for which all mappings
$\Ga\ni\ga\mapsto \sum_{x\in\ga} f(x)\in{\R}$ are continuous for any continuous function
$f$ on $\X$ with compact support. The space $\Ga$ with the vague topology is a Polish
space (see, e.g., \cite{KK2006} and references therein). The corresponding Borel $\sigma
$-algebra $\B(\Ga )$ will be the smallest $\sigma$-algebra for which all mappings $\Ga
\ni
\ga \mapsto |\ga_ \La |\in{ \N}_0:={\N}\cup\{0\}$ are measurable for any $\La\in{
\B}_{\mathrm{b}}(\X)$, see, e.g., \cite{AKR1998a}. We set ${{\mathcal{
F}}_{\mathrm{cyl}}}(\Ga )$ for the class of all \textit{cylinder functions} on $\Ga$.
Each
$F\in {{\mathcal{F}}_{\mathrm{cyl}}}(\Ga )$ is characterized by the following relation:
$F(\ga )=F(\ga_\La )$ for some $\La\in \B_{\mathrm{b}}(\X)$.

Let $0\leq\phi\in L^{1}(\R^{d})$ be given even function such
that
\begin{equation}\label{int_cond}
c_{\phi }:=\int_\X \Bigl( 1-e^{-\phi  ( x ) }\Bigr)
dx\in(0;\,+\infty).
\end{equation}
For any even $0\leq f\in L^1(\X)$ we denote
\begin{eqnarray*}
E^{f}(\eta)&:=&\sum_{x\in\eta}\sum_{y\in\eta\setminus x}f(x-y),
\qquad\eta\in\Ga_{0}\\
E^{f}(x,\ga)&:=&\sum_{y\in\ga\setminus x}f(x-y),
\qquad\ga\in\Ga, \,
x\in\X,\\
\langle f \rangle&:=&\int_\X f(x)dx.
\end{eqnarray*}

As it was already mentioned in the Introduction we would like to study two classes of
the
interacting particle systems (IPS), whose mechanisms of evolution are described by the
corresponding heuristically given Markov generators:
\begin{eqnarray}
(\Le F)(\ga)&:=&
m\sum_{x\in\ga}\bigl[F(\ga\setminus x)-F(\ga)\bigr]\nonumber\\
&&+\sum_{y\in\ga}\int_{\R^{d}}b_{0}( x,y,\gamma \setminus
y)
e^{-E^{\phi }( x,\gamma ) }\bigl[F(\ga\cup
x)-F(\ga)\bigr]dx\label{est-gen}
\end{eqnarray}
and
\begin{eqnarray}
(\Lf
F)(\ga)&:=&m\sum_{x\in\ga}\bigl[F(\ga\setminus
x)-F(\ga)\bigr]\nonumber\\
&&+\sum_{y\in\ga}e^{-E^{\phi } ( y,\gamma\setminus y )
}\int_{\R^{d}}b_0( x,y,\gamma \setminus y) \bigl[F(\ga\cup
x)-F(\ga)\bigr]dx.\label{fec-gen}
\end{eqnarray}
The first model shows the influence of establishment in the system and the second one
presents fecundity. Here and in the sequel the mortality $m$ is always supposed to be
strictly positive. One can see that the establishment rate $e^{-E^{\phi }( x,\gamma ) }$
will be smaller if $x$ will be inside or close to the dense region of the configuration
$\ga$. In its turn the fecundity rate $e^{-E^{\phi }( y,\gamma\setminus y ) }$ would be
also smaller if $y$ is situated in the dense area of $\ga$. The non-negative measurable
rate $b_0$ represents the dispersion of the model. Let $0\leq a^+, b^+\in L^1(\X)$ be
given even functions, and $\langle a^+\rangle=1$. We consider two types of the
dispersion:
\begin{itemize}
\item \emph{density independent dispersion}
\[
b_0(x,y,\gamma \setminus y)=\varkappa^+a^+(x-y),
\]
\item \emph{density dependent dispersion}
\[
b_0(x,y,\gamma \setminus y)=a^+(x-y)\biggl(
\varkappa^++\sum_{y'\in\ga\setminus y} b^+(y-y')\biggr).
\]
\end{itemize}

As it was mentioned above, we will study evolution of our model in terms of its
correlation functions. Below we introduce some basic notions needed to describe the
corresponding evolution.

The space of $n$-point configurations in an arbitrary $Y\in\B(\X)$ is defined by
\[
\Ga^{(n)}(Y):=\Bigl\{  \eta \subset Y \Bigm| |\eta |=n\Bigr\}
,\quad
n\in { \N}.
\]
By definition we take $\Ga^{(0)}(Y):=\{\emptyset\}$. As a set, $\Ga^{(n)}(Y)$ may be
identified with the symmetrization of $\widetilde{Y^n} = \bigl\{ (x_1,\ldots ,x_n)\in
Y^n
\bigm| x_k\neq x_l \ \mathrm{if} \ k\neq l\bigr\}$. Hence one can introduce the
corresponding Borel $\sigma $-algebra, which we denote by $\B\bigl(\Ga^{(n)}(Y)\bigr)$.
The space of finite configurations in an arbitrary $Y\in\B(\X)$ is defined by
\[
\Ga_0(Y):=\bigsqcup_{n\in {\N}_0}\Ga^{(n)}(Y).
\]
This space is equipped with the topology of the disjoint union. On $\Ga_0(Y)$ we
consider
the corresponding Borel $\sigma $-algebra denoted by $\B \bigl(\Ga_0(Y)\bigr)$. In the
case of $Y=\X$ we will omit $Y$ in the notation. Namely, $\Ga_0:=\Ga_{0}(\X)$,
$\Ga^{(n)}:=\Ga^{(n)}(\X)$.

The restriction of the Lebesgue product measure $(dx)^n$ to $\bigl(\Ga^{(n)},
\B(\Ga^{(n)})\bigr)$ we denote by $m^{(n)}$. We set $m^{(0)}:=\delta_{\{\emptyset\}}$.
The
Lebesgue--Poisson measure $\la $ on $\Ga_0$ is defined by
\[
\la :=\sum_{n=0}^\infty \frac {1}{n!}m^{(n)}.
\]
For any $\La\in\B_{\mathrm{b}}(\X)$ the restriction of $\la$ to $\Ga
(\La):=\Ga_{0}(\La)$
will be also denoted by $\la $. The space $\bigl( \Ga, \B(\Ga)\bigr)$ can be obtained as
the projective limit of the family of spaces $\bigl\{\bigl( \Ga(\La),
\B(\Ga(\La))\bigr)\bigr\}_{\La \in \B_{\mathrm{b}} (\X)}$, see, e.g., \cite{AKR1998a}.
The
Poisson measure $\pi$ on $\bigl(\Ga ,\B(\Ga )\bigr)$ is given as the projective limit of
the family of measures $\{\pi^\La \}_{\La \in \B_{\mathrm{b}} (\X)}$, where $
\pi^\La:=e^{-m(\La)}\la $ is the probability measure on $\bigl( \Ga(\La),
\B(\Ga(\La))\bigr)$ and $m(\La)$ is the Lebesgue measure of $\La\in \B_{\mathrm{b}}
(\X)$;
see, e.g., \cite{AKR1998a}.

A set $M\in \B (\Ga_0)$ is called bounded if there exists $ \La \in \B_{\mathrm{b}}
(\X)$
and $N\in { \N}$ such that $M\subset \bigsqcup_{n=0}^N\Ga^{(n)}(\La)$. The set of
bounded
measurable functions with bounded support we denote by $ B_{\mathrm{bs}}(\Ga_0)$, i.e.,
$G\in B_{\mathrm{bs}}(\Ga_0)$ if $ G\upharpoonright_{\Ga_0\setminus M}=0$ for some
bounded
$M\in {\B }(\Ga_0)$. Any $\B(\Ga_0)$-measurable function $G$ on $ \Ga_0$, in fact, is
defined by a sequence of functions $\bigl\{G^{(n)}\bigr\}_{n\in{ \N}_0}$ where $G^{(n)}$
is a $\B(\Ga^{(n)})$-measurable function on $\Ga^{(n)}$. As usual, functions on $\Ga$
are
called {\em observables} and functions on $\Ga_0$ are called {\em quasi-observables}.

There exists a mapping from $B_{\mathrm{bs}} (\Ga_0)$ into ${{
\mathcal{F}}_{\mathrm{cyl}}}(\Ga )$, which plays the key role in our further
considerations. It has the following form
\begin{equation}
KG(\ga ):=\sum_{\eta \Subset \ga }G(\eta ), \quad \ga \in
\Ga,
\label{KT3.15}
\end{equation}
where $G\in B_{\mathrm{bs}}(\Ga_0)$, see, e.g.,
\cite{KK2002,Len1975,Len1975a}. The summation in
\eqref{KT3.15} is
taken over all finite subconfigurations $\eta\in\Ga_0$ of the
(infinite) configuration $\ga\in\Ga$; we denote this by the
symbol,
$\eta\Subset\ga $. The mapping $K$ is linear, positivity
preserving,
and invertible, with
\begin{equation}
K^{-1}F(\eta ):=\sum_{\xi \subset \eta }(-1)^{|\eta \setminus
\xi
|}F(\xi ),\quad \eta \in \Ga_0.  \label{k-1trans}
\end{equation}
Note that if function $F$ has special form
\[
F(\ga)=\sum_{x\in\ga}H(x,\ga\setminus x),
\]
where $H(x,\cdot)$ is defined point-wisely
at least on $\Ga_0$,
then, by direct computation,
\begin{equation}
\bigl(K^{-1}F\bigr)(\eta
)=\sum_{x\in\eta}\bigl(K^{-1}H(x,\cdot)\bigr)(\eta\setminus
x),\quad \eta \in \Ga_0.
\label{KinverseSpec}
\end{equation}
We set also
\[
(K_0 G)(\eta):=(KG)(\eta), \qquad\eta\in\Ga_0.
\]

Let us define, for any $\B(\X)$-measurable function $f$, the
so-called coherent state
\[
e_\la (f,\eta ):=\prod_{x\in \eta }f(x) ,\ \eta \in \Ga
_0\!\setminus\!\{\emptyset\},\quad  e_\la (f,\emptyset ):=1.
\]
Then
\begin{equation}\label{Kexp}
(K_0e_\la (f))(\eta)=e_\la(f+1,\eta), \quad \eta\in\Ga_0
\end{equation}
and for any $f\in L^1(\X,dx)$
\begin{equation}\label{LP-exp-mean}
\int_{\Ga_0}e_\la (f,\eta)d\la(\eta)=\exp\Bigl\{\int_\X
f(x)dx\Bigr\}.
\end{equation}

A measure $\mu \in {\mathcal{M}}_{\mathrm{fm} }^1(\Ga )$ is called locally absolutely
continuous with respect to the Poisson measure $\pi$ if for any $\La \in \B_{\mathrm{b}}
(\X)$ the projection of $\mu$ onto $\Ga(\La)$ is absolutely continuous with respect to
the
projection of $ \pi$ onto $\Ga(\La)$. By \cite{KK2002}, in this case, there exists a
\emph{correlation functional} $k_{\mu}:\Ga_0 \rightarrow {\R}_+$ such that for any $G\in
B_{\mathrm{bs}} (\Ga_0)$ the following equality holds
\[
\int_\Ga (KG)(\ga) d\mu(\ga)=\int_{\Ga_0}G(\eta)
k_\mu(\eta)d\la(\eta).
\]
The restrictions $k_\mu^{(n)}$ of this functional on
$\Ga_0^{(n)}$,
$n\in{ \N}_0$ are called \emph{correlation functions} of the
measure
$\mu$. Note that $k_\mu^{(0)}=k_\mu(\emptyset)=1$.

We recall now without a proof the partial case of the well-known technical lemma (see
e.g. \cite{KMZ2004}) which plays very important role in our calculations.

\begin{lemma}
For any measurable function
$H:\Ga_0\times\Ga_0\times
\Ga_0\rightarrow{\R}$
\begin{equation}  \label{minlosid}
\int_{\Ga _{0}}\sum_{\xi \subset \eta }H\left( \xi ,\eta
\setminus
\xi ,\eta \right) d\la \left( \eta \right) =\int_{\Ga
_{0}}\int_{\Ga
_{0}}H\left( \xi ,\eta ,\eta \cup \xi \right) d\la \left( \xi
\right) d\la \left( \eta \right)
\end{equation}
if  both sides of the equality make sense.
\end{lemma}

For arbitrary and fixed $C>1$ we consider the functional Banach space
\begin{equation}
\mathcal{L}_{C}:=L^{1}(\Gamma _{0},C^{|\eta |}\lambda (d\eta )).
\label{space1}
\end{equation}
In the sequel, symbol $\left\Vert \cdot \right\Vert _{C}$ stands for the norm of the
space
\eqref{space1}.

Let $d\lambda _{C}:=C^{|\cdot |}d\lambda $, then the
dual space
\[
(\mathcal{L}_{C})'=\bigl(L^{1}(\Gamma _{0},d\lambda
_{C})\bigr)'=L^{\infty }(\Gamma _{0},d\lambda _{C}).
\]
The space $(\L_{C})'$ is isometrically isomorphic to the Banach space
\[
{\mathcal{K}}_{C}:=\left\{ k:\Gamma _{0}\rightarrow {\mathbb{R}}\,\Bigm| k\, C^{-|\cdot
|}\in L^{\infty }(\Gamma _{0},\lambda )\right\}
\]
with the norm
$$
\Vert k\Vert _{{\mathcal{K}}_{C}}:=\Vert C^{-|\cdot
|}k(\cdot )\Vert _{L^{\infty }(\Gamma _{0},\lambda )}
$$
where the
isomorphism is provided by the isometry $R_{C}$
\[
(\mathcal{L}_{C})'\ni k\longmapsto R_{C}k:=k\, C^{|\cdot
|}\in {\mathcal{K}}_{C}.
\]

In fact, one may consider the duality between the Banach
spaces $\mathcal{L}_{C}$ and ${\mathcal{K}}_{C}$ given by the following
expression
\begin{equation}
\left\langle \!\left\langle G,\,k\right\rangle
\!\right\rangle
:=\int_{\Gamma _{0}}G\cdot k\,d\lambda ,\quad G\in
\mathcal{L}_{C},\ k\in {\mathcal{K}}_{C}  \label{duality}
\end{equation}
with $\left\vert \left\langle \!\left\langle G,k\right\rangle
\!\right\rangle \right\vert \leq \Vert G\Vert _{C}\cdot \Vert
k\Vert _{{\mathcal{K}}_{C}}$. It is clear that $k\in {\mathcal{K}}_{C}$
implies
\[
|k(\eta )|\leq \Vert k\Vert _{{\mathcal{K}}_{C}}\,C^{|\eta|}
\qquad
\mathrm{for} \ \lambda\mathrm{-a.a.} \ \eta \in \Gamma _{0}.
\]

In the paper \cite{KKM2008}, it was proposed the analytic approach for the construction
of
non-equilibrium dynamics on $\Ga$, which uses deeply the harmonic analysis on
configuration spaces. By this approach the dynamics of correlation functions
corresponding
to \eqref{FokkerPlanck} is given by the evolutional equation
\begin{equation}
\frac{dk_t}{dt}=L^\triangle k_t, \qquad
k_t\bigm|_{t=0}=k_0\label{corfuncevol},
\end{equation}
where $L^\triangle$ is a dual operator to the $K$-image of $L$ defined by the expression
\[
\hat{L}:=K^{-1}LK
\]
with respect to the duality \eqref{duality}. Hence, $L^\triangle=\hat{L}^\ast$. In order
to construct the evolution of correlation functions we are going to follow such a
scheme:
we show that $\hat{L}$ is a generator of a $C_0$-semigroup in the certain Banach space
and
after consider the dual semigroup which solves the Cauchy problem \eqref{corfuncevol}.

\section{Functional evolutions}

Let
\[
\D :=\bigl\{ G\in \mathcal{L}_{C} \bigm|| \cdot | G(\cdot)\in
\mathcal{L}_{C}\bigr\} .
\]
Note that $B_\mathrm{bs}(\Ga_0) \subset \D$. In particular,
$\D$ is a dense set in $\L_C$.

In \cite{FKK2011a}, we have found sufficient conditions for operator $(\hat{L},\D)$ to
be
a generator of a semigroup in $\LC$. In the case of Markov generators \eqref{est-gen} or
\eqref{fec-gen}, this result may be formulated in the following way.
\begin{lemma}[{Theorem 3.2 of \cite{FKK2011a}}]
Suppose there exists $0<a<\frac{C}{2}$ such
that
\begin{equation}\label{sufcond}
\sum_{x\in\xi} \int_{\Gamma _{0}}\left\vert \left(
K_{0}^{-1}b\left( x,(\xi\setminus
x) \cup \cdot \right) \right)
\left( \eta \right) \right\vert C^{\left\vert
\eta \right\vert }d\lambda \left( \eta \right)
\leq am |\xi|,
\end{equation}
where $b(x,\eta)$ is equal either
\[
e^{-E^{\phi }( x,\eta ) } \sum_{y\in\eta}b_{0}( x,y,\eta \setminus y)\quad \mathit{or}
\quad \sum_{y\in\eta}e^{-E^{\phi } ( y,\eta\setminus y )}b_0( x,y,\eta \setminus y) .
\]
Then $(\hat{L},\D)$ is the generator of
a holomorphic semigroup in $\LC$.
\end{lemma}

It is worth noting that if \eqref{sufcond} is valid, then for any $G\in\D$
\begin{equation}\label{L_desc_expr}
\bigl( \hat{L}G\bigr) \left( \eta \right) =-m|\eta| G(\eta) + \sum_{\xi \subset \eta
}\int_{\mathbb{R}^{d}}\,G(\xi \cup x)\bigl(K_0^{-1}b(x,\cdot\cup\xi)\bigr)(\eta\setminus
\xi) dx. \end{equation}

\begin{theorem}\label{th-est}
Let $0\leq a^+, b^+, \phi \in L^{1}(\R^{d})$ be even functions such that
\eqref{int_cond}
holds and $\langle a^+\rangle=1$, $B:=\langle b^+\rangle\geq0$. Suppose, additionally,
that there exist constants $A_{1},A_{2}\geq0$ such that
\begin{eqnarray}
0\leq a^+(x)&\leq& A_{1}\phi(x), \qquad
x\in\X,\label{subdominate}\\
a^+(x-y)b^+(y-y')&\leq&
A_{2}\phi(x-y)\phi(x-y'), \qquad x,y,y'\in\X,
\label{specsubdominate}
\\
 \frac{A_{1}\varkappa^+}{eC}
+\frac{4A_{2}}{e^2C} &+&
\frac{A_1 B}{e} + \varkappa^+  + \frac{A_2\langle \phi
\rangle}{e} + C^{}B< \frac{m}{2}e^{-c_{\phi }C}
.\label{bigmort}
\end{eqnarray}
Then \eqref{sufcond}
holds and $\bigl(\hLe =K^{-1}\Le K,\D)$ is the generator of
a holomorphic semigroup $\hat{U}_\mathrm{est}(t)$
in $\LC$.
\end{theorem}
\begin{remark}
In the density independent case, $b^+\equiv0$, the assumption \eqref{specsubdominate}
holds with $A_2=0$. Moreover, since
$B=0$, the condition \eqref{bigmort} will have the following form
\[
 \frac{A_{1}\varkappa^+}{eC} +\varkappa^+ < \frac{m}{2}e^{-c_{\phi }C}.
\]
\end{remark}

Before proof of Theorem~\ref{th-est}, we give an example of $a^+$, $b^+$ which satisfy
\eqref{specsubdominate} in the Lemma below.
\begin{lemma}\label{example-sd}
Suppose that
there exist constants $E_{1},E_{2}>0$ and $\delta>d$ such
that
\[
  a^+( x ) \leq \frac{E_{1}}{\left( 1+\left\vert x\right\vert \right) ^{2\delta  }},
  \qquad {b^+} ( x )   \leq \frac{E_{1}}{\left( 1+\left\vert x\right\vert \right)
  ^{\delta
  }}\leq E_2\phi(x), \qquad x\in\X.
\]
Then \eqref{specsubdominate} holds with $A_2=E_2^2$.
\end{lemma}
\begin{proof}[Proof of Lemma \ref{example-sd}] Using obvious
inequality
\[
1+\left\vert x-y'\right\vert \leq 1+\left\vert x-y\right\vert
+\left\vert y-y'\right\vert \leq \left( 1+\left\vert
x-y\right\vert
\right) \left( 1+\left\vert y-y'\right\vert \right)
\]
we obtain that
\[
a^+ ( x-y )  {b^+}\left( y-y'\right) \leq \frac{E_{1}}{\left(
1+\left\vert x-y\right\vert \right) ^{\delta
}}\frac{E_{1}}{\left(
1+\left\vert x-y'\right\vert \right) ^{\delta }}\leq
E_2^2\phi(x-y)\phi(x-y'),
\]
that proves the statement.
\end{proof}

\begin{proof}[Proof of Theorem \ref{th-est}]
Let us set
\begin{equation}\label{b_est-expr}
b_{\mathrm{est}}(x,\ga)=e^{-E^{\phi }( x,\gamma )
}\sum_{y\in\ga}a^+(x-y)\biggl(
\varkappa^++\sum_{y'\in\ga\setminus y} b^+(y-y')\biggr).
\end{equation}
To check \eqref{sufcond}, we will try to estimate the integral \[
\int_{\Gamma _{0}}\left\vert \left( K_{0}^{-1}b_{\mathrm{est}}\left( x,\xi \cup \cdot
\right) \right) \left( \eta \right) \right\vert C^{\left\vert \eta \right\vert }d\lambda
\left( \eta \right), \qquad \xi\in\Ga_0
\]
uniformly in $x\in\X$ and $\xi\in\Ga_0$.

In view of \eqref{b_est-expr}, one has
\begin{eqnarray*}
b_{\mathrm{est}}\left( x,\xi \cup \eta \right)
&=& e^{-E^{\phi
}(x,\xi )}e^{-E^{\phi }(x,\eta )}\sum_{y\in \xi
}a^{+}(x-y)\biggl(
\varkappa ^{+}+\sum_{y'\in \xi \setminus
y}b^{+}(y-y')\biggr) \nonumber\\
&&+e^{-E^{\phi }(x,\xi )}e^{-E^{\phi }(x,\eta )}\sum_{y'\in
\eta
}\sum_{y\in \xi }a^{+}(x-y)b^{+}(y-y') \nonumber\\
&&+e^{-E^{\phi }(x,\xi )}e^{-E^{\phi }(x,\eta )}\sum_{y'\in
\eta
}a^{+}(x-y')\biggl( \varkappa ^{+}+\sum_{y\in \xi
}b^{+}(y-y')\biggr) \nonumber\\
&&+e^{-E^{\phi }(x,\xi )}e^{-E^{\phi }(x,\eta )}\sum_{y\in
\eta
}a^{+}(x-y)\sum_{y'\in \eta \setminus y}b^{+}(y-y').\nonumber
\end{eqnarray*}
Using \eqref{k-1trans}--\eqref{Kexp}, we obtain
\begin{eqnarray}
  &&\left( K_{0}^{-1}b_{\mathrm{est}}\left( x,\xi \cup \cdot \right) \right) \left( \eta
  \right) \label{term-est}\\ &=&\,e_{\lambda }\left( e^{-\phi (x-\cdot) }-1,\eta \right)
  b_{\mathrm{est}}\left( x,\xi \right) \nonumber \\ &&+e^{-E^{\phi }(x,\xi )}\sum_{y'\in
  \eta }\sum_{y\in \xi }a^{+}(x-y)b^{+}(y-y')e^{-\phi (x-y') }e_{\lambda }\left(
  e^{-\phi
  (x-\cdot) }-1,\eta \setminus y'\right) \nonumber \\ &&+e^{-E^{\phi }(x,\xi
  )}\sum_{y'\in
  \eta }e_{\lambda }\left( e^{-\phi (x-\cdot) }-1,\eta \setminus y'\right)
  a^{+}(x-y^{\prime })e^{-\phi (x-y')}\nonumber \\&&\qquad\qquad\times \biggl( \varkappa
  ^{+}+\sum_{y\in \xi }b^{+}(y-y')\biggr) \nonumber \\ &&+e^{-E^{\phi }(x,\xi )}\sum_{y
  \in \eta }\sum_{y' \in \eta\setminus y }a^{+}(x-y)b^{+}(y-y')e^{-\phi (x-y) }e^{-\phi
  (x-y') }\nonumber \\&&\qquad\qquad\times e_{\lambda }\left( e^{-\phi (x-\cdot)
  }-1,\eta
  \setminus \left\{ y,y'\right\} \right).\nonumber
\end{eqnarray}
Next, let $\kappa =e^{c_{\phi }C}$ then, by \eqref{LP-exp-mean},
\begin{eqnarray*}
  &&\int_{\Gamma _{0}}\left\vert \left( K_{0}^{-1}b_{\mathrm{est}}\left( x,\xi \cup
  \cdot
  \right) \right) \left( \eta \right) \right\vert C^{\left\vert \eta \right\vert
  }d\lambda
  \left( \eta \right)\nonumber  \\ &\leq &\,\kappa b_{\mathrm{est}}\left( x,\xi \right)
  +Ce^{-E^{\phi }(x,\xi )}\int_{\Gamma _{0}}\int_{\mathbb{R}^{d}}\sum_{y\in \xi
  }a^{+}(x-y)b^{+}(y-y')e^{-\phi (x-y') } \nonumber \\ &&\qquad\qquad\times e_{\lambda
  }\left( \left\vert e^{-\phi (x-\cdot) }-1\right\vert ,\eta \right) C^{\left\vert \eta
  \right\vert }dy'd\lambda \left( \eta \right) \nonumber \\ &&+Ce^{-E^{\phi }(x,\xi
  )}\int_{\Gamma _{0}}\int_{\mathbb{R}^{d}}e_{\lambda }\left( \left\vert e^{-\phi
  (x-\cdot) }-1\right\vert ,\eta \right) a^{+}(x-y')e^{-\phi (x-y') } \nonumber \\
  &&\qquad\qquad\times \biggl( \varkappa ^{+}+\sum_{y\in \xi }b^{+}(y-y')\biggr)
  C^{\left\vert \eta \right\vert }dy'd\lambda \left( \eta \right) \nonumber \\
  &&+C^{2}e^{-E^{\phi }(x,\xi )}\int_{\Gamma _{0}}\int_{\mathbb{R}^{d}}\int_{
  \mathbb{R}^{d}}a^{+}(x-y)b^{+}(y-y')e^{-\phi (x-y) }e^{-\phi (x-y') }\nonumber  \\
  &&\qquad\qquad\times e_{\lambda }\left( \left\vert e^{-\phi (x-\cdot) }-1\right\vert
  ,\eta \right) C^{\left\vert \eta \right\vert }dy'dyd\lambda \left( \eta \right)
  \nonumber \\ &\leq &\,\kappa b_{\mathrm{est}}\left( x,\xi \right) +\kappa B
  Ce^{-E^{\phi
  }(x,\xi )}\sum_{y\in \xi }a^{+}(x-y) \nonumber \\ &&+\kappa Ce^{-E^{\phi }(x,\xi
  )}\varkappa ^{+}\left\langle a^{+}e^{-\phi }\right\rangle \nonumber \\ && +\kappa
  Ce^{-E^{\phi }(x,\xi )}\sum_{y\in \xi }\int_{\mathbb{R }^{d}}a^{+}(x-y')e^{-\phi
  (x-y')
  }b^{+}(y-y')dy' \nonumber \\ &&+\kappa C^{2}e^{-E^{\phi }(x,\xi
  )}\int_{\mathbb{R}^{d}}\int_{\mathbb{R} ^{d}}a^{+}(x-y)b^{+}(y-y')e^{-\phi (x-y)
  }e^{-\phi (x-y') }dy'dy.\nonumber
\end{eqnarray*}
By \eqref{subdominate}, one has
\[
e^{-E^{\phi }(x,\xi )} \sum_{y\in \xi
}a^{+}(x-y)\leq A _1e^{-E^{\phi }(x,\xi )} E^{\phi }(x,\xi
)\leq \frac{A_1}{e},
\]
where we used the elementary inequality $xe^{-x}\leq e^{-1}$,
$x\geq0$. Next, by \eqref{specsubdominate}, we may estimate
\begin{eqnarray*}
&&e^{-E^{\phi }(x,\xi )}\sum_{y\in \xi }\int_{\mathbb{R
}^{d}}a^{+}(x-y') e^{-\phi (x-y') }b^{+}(y-y')dy'\\&\leq &
\, A_2e^{-E^{\phi }(x,\xi )}\sum_{y\in \xi }\int_\X\phi(x-y)\phi(x-y')dy'\leq
\frac{A_2\langle\phi\rangle}{e}.
\end{eqnarray*}
Moreover, \eqref{subdominate}, \eqref{specsubdominate} yield
\begin{eqnarray}
b_{\mathrm{est}}(x,\xi)&\leq &\, A_1 \varkappa^+e^{-E^{\phi }(
x,\xi ) }\sum_{y\in\xi}\phi(x-y)\nonumber
 \\ &&+A_2 e^{-E^{\phi }( x,\xi )
 }\sum_{y\in\xi}\phi(x-y)\sum_{y'\in\xi\setminus y}
 \phi(x-y')\nonumber\\
&\leq& \frac{A_1\varkappa^+}{e}+A_{2} e^{-E^{\phi }( x,\xi )
}\bigl((E^{\phi }( x,\xi )\bigr)^2\leq \frac{A_{1}\varkappa^+}{e}
+\frac{4A_{2}}{e^2},\nonumber
\end{eqnarray}
since $x^2e^{-x}\leq 4e^{-2}$, $x\geq0$.

Therefore, we have
\begin{eqnarray*}
&&\int_{\Gamma _{0}}\left\vert \left(
K_{0}^{-1}b_{\mathrm{est}}\left( x,\xi \cup \cdot \right)
\right)
\left( \eta \right) \right\vert C^{\left\vert
\eta \right\vert }d\lambda \left( \eta \right)\\
&\leq &\kappa \biggl(\frac{A_{1}\varkappa^+}{e}
+\frac{4A_{2}}{e^2}\biggr) +\kappa
CB\frac{A_1}{e} +\kappa C\varkappa^+  +\kappa
C\frac{A_2\langle \phi \rangle}{e} +\kappa C^{2}B=:D.
\end{eqnarray*}
To obtain \eqref{sufcond}, it
is enough to suppose that $D\leq a m$, where
$\frac{a}{C}<\frac{1}{2}$. Hence, we need that
$m>\frac{2D}{C}$ only, that is \eqref{bigmort}.
The theorem is proved.
\end{proof}

\begin{theorem}\label{th-fec}
Let $0\leq a^+, b^+, \phi \in L^{1}(\R^{d})$ be even functions such that
\eqref{int_cond}
holds and $\langle a^+\rangle=1$, $B=\langle b^+\rangle\geq0$. Suppose, additionally,
that
there exists constants
$A_{1},A_{2}\geq0$
such that for a.a. $x,y,y'\in\X$
\begin{eqnarray}
0\leq a^+(x)&\leq&
A_{1}\phi(x)e^{-\phi(x)},\label{subdominate-fec}\\
b^+(x)&\leq& A_2 \phi(x), \label{subdominate-fec-b+}\\
\varkappa ^{+}+\frac{
 A_2}{e}&+& C B+\Bigl(\frac{\varkappa ^{+}}{C}+
B\Bigr) \frac{A_1}{e} +\frac{4A_1A_2}{e^2}C<
\frac{m}{2}e^{-c_{\phi }C} .\label{big-mort-fec}
\end{eqnarray}
Then \eqref{sufcond}
holds and $\bigl(\hLf =K^{-1}\Lf K,\D)$ is the generator of
a holomorphic semigroup $\hat{U}_\mathrm{fec}(t)$
in $\LC$.
\end{theorem}
\begin{remark}
In the density independent case, $A_2=B=0$,
and \eqref{big-mort-fec} may be rewritten
in the form:
\[
\varkappa ^{+}\Bigl(1+\frac{A_1}{eC}\Bigr) <
\frac{m}{2}e^{-c_{\phi }C}.
\]
\end{remark}
\begin{proof}
Set
\[
b_{\mathrm{fec}}(x,\ga)=\sum_{y\in\ga}e^{-E^{\phi } (
y,\gamma\setminus y )
}b_0( x,y,\gamma \setminus y).
\]
Then, one has
\begin{eqnarray*}
b_{\mathrm{fec}}\left( x,\eta \cup \xi \right)&=&\sum_{y\in
\eta
}e^{-E^{\phi }(y,\xi )}e^{-E^{\phi }(y,\eta \setminus
y)}a^{+}(x-y)\biggl( \varkappa ^{+}+\sum_{y'\in \xi
}b^{+}(y-y')\biggr)  \nonumber\\
&&+\sum_{y\in \eta }e^{-E^{\phi }(y,\xi )}e^{-E^{\phi }(y,\eta
\setminus
y)}a^{+}(x-y)\sum_{y'\in \eta \setminus y}b^{+}(y-y')
\nonumber\\
&&+\sum_{y\in \xi }e^{-E^{\phi }(y,\eta )}e^{-E^{\phi }(y,\xi
\setminus y)}a^{+}(x-y)\biggl( \varkappa ^{+}+\sum_{y'\in \xi
\setminus
y}b^{+}(y-y')\biggr)  \nonumber\\
&&+\sum_{y'\in \eta }\sum_{y\in \xi }e^{-E^{\phi }(y,\eta
)}e^{-E^{\phi }(y,\xi \setminus
y)}a^{+}(x-y)b^{+}(y-y'),\nonumber
\end{eqnarray*}
and, using \eqref{k-1trans}--\eqref{Kexp}, we obtain
\begin{eqnarray}\label{term-fec}
&&\left( K_{0}^{-1}b_{\mathrm{fec}}\left( x,\xi \cup \cdot
\right)
\right) \left( \eta \right)  \\&=& \,\sum_{y\in \eta
}e^{-E^{\phi
}(y,\xi )}e_{\lambda }\left( e^{-\phi \left( y-\cdot \right)
}-1,\eta \setminus y\right) a^{+}(x-y)\biggl( \varkappa
^{+}+\sum_{y'\in \xi }b^{+}(y-y')\biggr)  \nonumber\\
&&+\sum_{y\in \eta }e^{-E^{\phi }(y,\xi )}a^{+}(x-y)\sum_{y'\in
\eta
\setminus y}b^{+}(y-y')e^{-\phi \left( y-y'\right)
}e_{\lambda
}\left( e^{-\phi \left( y-\cdot \right) }-1,\eta \setminus
y\setminus y'\right)  \nonumber\\
&&+\sum_{y\in \xi }e_{\lambda }\left( e^{-\phi \left( y-\cdot
\right)
}-1,\eta \right) e^{-E^{\phi }(y,\xi \setminus
y)}a^{+}(x-y)\biggl(
\varkappa
^{+}+\sum_{y'\in \xi \setminus y}b^{+}(y-y')\biggr)
\nonumber\\
&&+\sum_{y'\in \eta }e_{\lambda }\left( e^{-\phi \left(
y-\cdot
\right) }-1,\eta \setminus y'\right) e^{-\phi \left(
y-y'\right)
}\sum_{y\in \xi }e^{-E^{\phi }(y,\xi \setminus
y)}a^{+}(x-y)b^{+}(y-y').\nonumber
\end{eqnarray}
Therefore, for $\kappa =e^{c_{\phi }C}$ we
have, by \eqref{LP-exp-mean},
 \begin{eqnarray*}
&&\int_{\Gamma _{0}}\left\vert \left(
K_{0}^{-1}b_{\mathrm{fec}}\left( x,\xi \cup \cdot \right)
\right)
\left( \eta \right) \right\vert C^{\left\vert
\eta \right\vert }d\lambda \left( \eta \right)  \\
&\leq &\,\kappa C\int_{\mathbb{R}^{d}}e^{-E^{\phi }(y,\xi
)}a^{+}(x-y)\biggl(
\varkappa ^{+}+\sum_{y'\in \xi }b^{+}(y-y')\biggr) dy \\
&&+\kappa C^{2}\int_{\mathbb{R}^{d}}\int_{\mathbb{R}
^{d}} e^{-E^{\phi }(y,\xi )}a^{+}(x-y)b^{+}(y-y')e^{-\phi
\left(
y-y'\right)
}dydy' \\
&&+\kappa \sum_{y\in \xi }e^{-E^{\phi }(y,\xi \setminus
y)}a^{+}(x-y)\biggl( \varkappa ^{+}+\sum_{y'\in \xi \setminus
y}b^{+}(y-y')\biggr)  \\
&&+\kappa C\int_{\mathbb{R}^{d}}e^{-\phi \left( y-y'\right)
}\sum_{y\in \xi }e^{-E^{\phi }(y,\xi \setminus
y)}a^{+}(x-y)b^{+}(y-y')dy'\\
&\leq&\,\kappa C\varkappa ^{+}+\kappa
C\int_{\mathbb{R}^{d}}a^{+}(x-y)e^{-E^{\phi }(y,\xi
)}\sum_{y'\in
\xi }
b^{+}(y-y') dy \\
&&+\kappa C^{2}B+(\kappa\varkappa ^{+}+\kappa
CB) \sum_{y\in \xi }e^{-E^{\phi }(y,\xi \setminus
y)}a^{+}(x-y) \\
&&+\kappa \sum_{y\in \xi }e^{-E^{\phi }(y,\xi \setminus
y)}a^{+}(x-y)
\sum_{y'\in \xi \setminus y}b^{+}(y-y').
\end{eqnarray*}
By \eqref{subdominate-fec-b+},
\[
\kappa C\int_{\mathbb{R}^{d}}a^{+}(x-y)e^{-E^{\phi }(y,\xi
)}\sum_{y'\in \xi } b^{+}(y-y') dy \leq \frac{\kappa C
A_2}{e}.
\]
Now we will verify the sufficient condition \eqref{sufcond}. By \eqref{subdominate-fec},
one
has
\begin{eqnarray*}
&&\sum_{x\in\xi}\sum_{y\in \xi\setminus x }e^{-E^{\phi }(y,(\xi
\setminus x)\setminus
y)}a^{+}(x-y)=\sum_{y\in \xi }\sum_{x\in\xi\setminus
y}e^{-E^{\phi }(y,(\xi \setminus y)\setminus
x)}a^{+}(x-y)\\ &=&\sum_{y\in \xi }e^{-E^{\phi }(y,\xi
\setminus y)}\sum_{x\in\xi\setminus y}e^{\phi(x-y)}a^{+}(x-y)
\leq \frac{A_1}{e}|\xi|,
\end{eqnarray*}
and, by \eqref{subdominate-fec} and
\eqref{subdominate-fec-b+}
we have
\begin{eqnarray*}
&&\sum_{x\in\xi} \sum_{y\in \xi \setminus x}e^{-E^{\phi
}(y,(\xi\setminus x) \setminus y)}a^{+}(x-y) \sum_{y'\in (\xi
\setminus x) \setminus y}b^{+}(y-y')\\ &=&\sum_{y\in\xi}
e^{-E^{\phi
}(y,\xi\setminus y)}\sum_{x\in \xi \setminus
y}e^{\phi(x-y)}a^{+}(x-y) \sum_{y'\in (\xi \setminus y)
\setminus
x}b^{+}(y-y')\\&\leq &\,A_1 A_{2}\sum_{y\in\xi} e^{-E^{\phi
}(y,\xi\setminus y)}\sum_{x\in \xi \setminus y}\phi(x-y)
\sum_{y'\in
\xi \setminus y}\phi(y-y')\leq \frac{4A_1A_2}{e^2}|\xi|.
\end{eqnarray*}
Hence, by \eqref{big-mort-fec}, we obtain
\eqref{sufcond}.
\end{proof}

Let $\bigl({\hat{L}}',\Dom({\hat{L}}')\bigr)$ be an operator
in $(\mathcal{L}_{C})'$ which is dual to the closed operator
$\bigl(
{\hat{L}},\D\bigr)$. Here and below $\hat{L}$
means either $\hLe$ or $\hLf$. We consider also its image on
${\mathcal{K}}_{C}$ under the isometry $R_{C}$, namely, let ${\hat{L}}^{\ast
}=R_{C}{\hat{L}}'R_{C^{-1}}$ with the domain $\Dom({\hat{L}}^{\ast
})=R_{C}\Dom({\hat{L}}')$.

By Proposition 3.5 of \cite{FKK2011a}, for
any $\alpha\in(0;1)$
\[
\mathcal{K}_{\alpha C}\subset \Dom({\hat{L}}^{\ast }).
\]

Under the conditions of Theorem~\ref{th-est} or Theorem~~\ref{th-fec}, there exists
$a\in\bigl(0;\frac{C}{2}\bigr)$ such that \eqref{sufcond} holds. In the following let
$\hat{T}(t)$ denotes either $\hat{U}_\mathrm{est}(t)$ or $\hat{U}_\mathrm{fec}(t)$. One
can consider the adjoint semigroup $\hat{T}'(t)$ in $(\mathcal{L}_{C})'$ and its image
$\hat{T}^{\ast }(t)$ in ${\mathcal{K}}_{C}$. By, e.g., Subsection~II.2.6 of \cite{EN2000},
the restriction $\hat{T}^{\odot }(t)$ of the semigroup $\hat{T}^{\ast }(t)$ onto its
invariant Banach subspace $\overline{\Dom({\hat{L}}^{\ast })}$ (here and below all
closures are in the norm of the space ${\mathcal{K}}_{C}$) is a strongly continuous
semigroup. Moreover,  its generator ${\hat{L}} ^\odot$ will be part of ${\hat{L}}^\ast$,
namely,
\[
\Dom({\hat{L}} ^\odot)=\Bigl\{k\in \Dom({\hat{L}}^\ast)
\Bigm|
{\hat{L}}^\ast k\in \overline{\Dom({\hat{L}}^\ast)}\Bigr\}
\]
and ${\hat{L} }^\ast k ={\hat{L}}^\odot k$ for any $k\in
\Dom({\hat{L}}^\odot)$.

\begin{theorem}[{Theorem 3.8 of \cite{FKK2011a}}]
For any $\alpha
\in \left( \frac{2a}{C};1\right) $ the set $\overline{
\mathcal{K}_{\alpha C}}$ is a $\hat{T}^{\odot }(t)$-invariant
Banach
subspace of $\mathcal{K}_{C}$.
\end{theorem}

Therefore, for $\alpha \in \left( \frac{2a}{C};1\right)$, one
can consider the restriction $\hat{T}^{\odot \alpha}$
of the semigroup $\hat{T}^{\odot}$ onto $\overline{\K_{\alpha
C}}$.
This restriction will be strongly continuous semigroup with
generator $\hat{L}^{\odot \alpha}$ which is restriction of
${\hat{L}}^\odot$ onto $\overline{\K_{\alpha C}}$ (see, e.g.,
Subsection II.2.3 of \cite{EN2000}). Therefore,
\[
\Dom({\hat{L}} ^{\odot\alpha})=\Bigl\{k\in
\overline{\K_{\alpha C}} \Bigm|
{\hat{L}}^\ast k\in \overline{\K_{\alpha C}}\Bigr\},
\]
and ${\hat{L}} ^{\odot\alpha}$ coincides
with $\hat{L}^\ast$ on $\Dom({\hat{L}} ^{\odot\alpha})$.
Note that for any $k\in\K_{\alpha C}\subset
D(\hat{L}^\ast)$
\[
\bigl(\hat{L}^\ast k\bigr)(\eta)=-m|\eta|k(\eta)|+\sum_{x\in
\eta }\int_{\Ga _{0}}k(\xi \cup (\eta \setminus
x))\bigl( K_0^{-1}b(x,\cdot\cup \eta\setminus x)\bigr)
(\xi)d\la
(\xi).
\]
The explicit expressions can be found using \eqref{term-est} or \eqref{term-fec}.

Hence, we have the strong solution (in the sense of the norm in $\K_C$) of the evolution
equation \begin{equation}\label{KFP-evoleqn} \frac{\partial}{\partial t} k_t =
\hat{L}^\ast k_t \end{equation} at least on the subspace $\K_{\alpha C}$.

\begin{remark}
To study stationary equation $\hat{L}^\ast k=0$ corresponding to \eqref{KFP-evoleqn} on
the set of functions $k\in\K_{\alpha C}$ such that $k(\emptyset)=1$, one may consider
even
weaker assumptions without denominator $2$ in \eqref{bigmort} or \eqref{big-mort-fec}.
However, by Proposition 3.9 of \cite{FKK2011a}, a~unique solution of this equation will
satisfy $k(\eta)=0$ for all $|\eta|\neq0$.
\end{remark}

\section{Vlasov scaling}
To begin with, we would like to explain the idea of
the Vlasov-type scaling. The general scheme describing this scaling for the
birth-and-death dynamics as well as for the conservative ones may be found in
\cite{FKK2010a}. This approach was successfully realized for the
Bolker--Dieckmann--Law--Pacala model \eqref{BP-gen}--\eqref{BP-birthrate}
in \cite{FKK2010c}.

Let us now detail how we proceed to organize the Vlasov-type scaling. We will initially
scale the generator $L$ by the scaling parameter $\eps>0$, in such a way that the
following holds. First of all the $K$-image $\hat{L}_\eps$
of the rescaled operator $L_\eps$ has to be a generator of a semigroup on
some $\L_{C_\eps}$. Consider the corresponding dual semigroup
$\hat{T}_\eps^\ast(t)$. Let us choose an initial function of the
corresponding Cauchy problem depending on $\eps$ in such a way that $k_0^\e(\eta) \sim
\eps^{-|\eta|} r_0(\eta)$,
$\eps\rightarrow 0$,
$\eta\in\Ga_0$ with some function $r_0$, independent of
$\eps$. Secondly, the scaling $L\mapsto L_\eps$ has to be performed to assure that the
semigroup $\hat{T}_\eps^\ast(t)$ preserves the order of the singularity:
\[
(\hat{T}_\eps^\ast(t)k_0^\e)(\eta) \sim \eps^{-|\eta|}
r_t(\eta),
\quad \eps\rightarrow 0, \ \ \eta\in\Ga_0.
\]
Moreover, the dynamics $r_0 \mapsto r_t$
should preserve coherent states. Namely, if
$r_0(\eta)=e_\la(\rho_0,\eta)$, then
$r_t(\eta)=e_\la(\rho_t,\eta)$
and there exists explicit (nonlinear, in general)
differential equation for $\rho_t$:
\[
\frac{\partial}{\partial t}\rho_t(x) = \upsilon(\rho_t)(x)
\]
which is called the Vlasov-type equation.

Below we realize this approach for the case of
\[
(LF)(\ga)=m
\sum_{x\in\ga}D_{x}^{-}F(\ga)+\int_{\R^{d}}b^{}(x,\ga)D_{x}^{+}F(\ga)dx,
\]
where $b=b(a^+,b^+,\phi)$ is either birth rate with establishment (see \eqref{est-gen})
or the one corresponding to the fecundity mechanism.
Let us consider for any $\eps\in(0;1]$ the following scaling
\[
(L_\eps F)(\ga)=m
\sum_{x\in\ga}D_{x}^{-}F(\ga)+\eps^{-1}\int_{\R^{d}}b^{}_\eps(x,\ga)D_{x}^{+}F(\ga)dx,
\]
with $b_\eps=b(\eps a^+,\eps b^+,\eps \phi)$. Here $D_x^\pm$
are given by \eqref{elementarydif}. We denote by
$b_{\eps,\mathrm{est}}$
and $b_{\eps,\mathrm{fec}}$ the scaled rates for the corresponding models.
We
define also the renormalized operator (see \cite{FKK2010a,FKK2010c} for details)
\[
\hat{L}_{\eps,\mathrm{ren}}:=R_{\eps^{-1}}K^{-1}L_\eps
KR_\eps,
\]
where $(R_\sigma G)(\eta)
 = \sigma^{|\eta|}G(\eta)$ for arbitrary $\sigma>0$.
\begin{lemma}\label{lemma-eps}
Suppose that the conditions of Theorem \ref{th-est}
(or Theorem \ref{th-fec}) are satisfied with $\langle\phi\rangle$
instead of $c_\phi$ in \eqref{bigmort}
(in \eqref{big-mort-fec}, correspondingly).  Then there exists
$a\in\Bigl(0;\frac{C}{2}\Bigr)$
such that
\begin{equation}\label{sufcond-eps}
\sum_{x\in\xi} \int_{\Gamma _{0}}\left\vert \left(
K_{0}^{-1}b_\eps\left( x,(\xi\setminus
x) \cup \cdot \right) \right)
\left( \eta \right) \right\vert \eps^{-|\eta|}
C^{\left\vert
\eta \right\vert }d\lambda \left( \eta \right)
\leq am |\xi|,
\end{equation}
where $b_\eps=b_{\eps,\mathrm{est}}$ (or
$b_\eps=b_{\eps,\mathrm{fec}}$,
correspondingly).
\end{lemma}
\begin{proof}
We begin with the establishment case. Set
\[
\psi_\eps(x)=\eps^{-1}\bigl(e^{-\eps \phi (x)}-1\bigr), \quad
x\in\X.
\]
By \eqref{term-est}, we have
\begin{eqnarray}
&&\eps^{-|\eta|}\left( K_{0}^{-1}b_{\eps,\mathrm{est}}\left(
x,\xi \cup \cdot \right)
\right) \left( \eta \right) \label{term-est-eps}\\
&=&\,\eps e_{\lambda }\left( \psi_\eps(x-\cdot),\eta \right)
e^{-\eps E^{\phi }( x,\xi ) }\sum_{y\in\xi} a^+(x-y)\biggl(
\varkappa^++\sum_{y'\in\xi\setminus y} \eps b^+(y-y')\biggr)
\nonumber \\
&&+\eps e^{-\eps E^{\phi }(x,\xi )}\sum_{y'\in \eta }\sum_{y\in
\xi
} a^{+}(x-y) b^{+}(y-y')e^{-\eps \phi (x-y') }e_{\lambda
}\left( \psi_\eps(x-\cdot),\eta \setminus
y'\right) \nonumber \\
&&+e^{-\eps E^{\phi }(x,\xi )}\sum_{y'\in \eta }e_{\lambda
}\left(
\psi_\eps(x-\cdot),\eta \setminus y'\right) a^{+}(x-y^{\prime
})e^{-\eps \phi (x-y')}\nonumber \\&&\qquad\qquad\times \biggl(
\varkappa
^{+}+\eps \sum_{y\in \xi }b^{+}(y-y')\biggr) \nonumber \\
&&+e^{-\eps E^{\phi }(x,\xi )}\sum_{y \in \eta
}\sum_{y' \in \eta\setminus y
}a^{+}(x-y) b^{+}(y-y')e^{-\eps \phi (x-y) }e^{-\eps \phi
(x-y')
}\nonumber \\&&\qquad\qquad\times e_{\lambda }\left(
\psi_\eps(x-\cdot),\eta \setminus \left\{ y,y'\right\}
\right).\nonumber
\end{eqnarray}
Since $\eps\in(0;1]$ and
\[
|\psi_\eps(x)|\leq \phi(x), \qquad
x\in\X,
\]
the estimate for $\eps^{-|\eta|}\left|
K_{0}^{-1}b_{\eps,\mathrm{est}}\left( x,\xi \cup \cdot
\right)
\right| \left( \eta \right)$
will be almost the same as for
$\left| K_{0}^{-1}b\left( x,\xi \cup \cdot \right)
\right|
\left( \eta \right)$
in the proof of Theorem~\ref{th-est}.
The changes will concern the term $|e^{-\phi}-1|$ which will be substitute by $\phi$.
This leads to the new constant $\langle\phi\rangle$ instead of $c_\phi$ in
further estimates. The rest part of the proof is the same as for the non-scaled case.

The same approach may be used for the case of fecundity. Indeed,
\begin{eqnarray}\label{term-fec-eps}
&&
\eps^{-|\eta|}\left( K_{0}^{-1}b_{\eps,\mathrm{fec}}\left(
x,\xi \cup \cdot \right)
\right) \left( \eta \right)  \\
&=& \,\sum_{y\in \eta }e^{-\eps
E^{\phi
}(y,\xi )}e_{\lambda }\left( \psi_\eps(y-\cdot),\eta \setminus
y\right)  a^{+}(x-y)\biggl( \varkappa
^{+}+\sum_{y'\in \xi }\eps b^{+}(y-y')\biggr)  \nonumber\\
&&+\sum_{y\in \eta }e^{-\eps E^{\phi }(y,\xi )}
a^{+}(x-y)\sum_{y'\in \eta
\setminus y} b^{+}(y-y')e^{-\eps \phi \left( y-y'\right)
}e_{\lambda
}\left(  \psi_\eps(y-\cdot),\eta \setminus
y\setminus y'\right)  \nonumber\\
&&+\eps\sum_{y\in \xi }e_{\lambda }\left(
\psi_\eps(y-\cdot),\eta \right) e^{-\eps E^{\phi }(y,\xi
\setminus y)} a^{+}(x-y)\biggl(
\varkappa
^{+}+\sum_{y'\in \xi \setminus y}\eps b^{+}(y-y')\biggr)
\nonumber\\
&&+\eps\sum_{y'\in \eta }e_{\lambda }\left(
\psi_\eps(y-\cdot),\eta \setminus y'\right) e^{-\eps \phi
\left( y-y'\right)
}\sum_{y\in \xi }e^{-\eps E^{\phi }(y,\xi \setminus
y)} a^{+}(x-y) b^{+}(y-y').\nonumber
\end{eqnarray}
The analogous arguments to establishment case complete the proof.
\end{proof}

Under conditions of Lemma \ref{lemma-eps} we have the
following result about the renormalized semigroups in $\LC$
and $\K_C$.
\begin{proposition}[{Proposition 4.1 of \cite{FKK2011a}}]
Let the conditions of Lemma~\ref{lemma-eps}
hold. Then for any $\eps\in(0;1]$,
$\bigl(\hat{L}_{\eps,\mathrm{est},\mathrm{ren}},
\D\bigr)$ and
$\bigl(\hat{L}_{\eps,\mathrm{fec},\mathrm{ren}},
\D\bigr)$ are the generators of holomorphic semigroups
$\hat{U}_{\eps,\mathrm{est}}(t)$ and
$\hat{U}_{\eps,\mathrm{fec}}(t)$
on $\L_C$, correspondingly.
Moreover, there exists $\alpha_0\in(0;\frac{1}{\nu})$ such
that for any
$\alpha\in(\alpha_0;\frac{1}{\nu})$ and  $\eps\in(0;1]$
there
exist a strongly continuous semigroups $\hat{U}^{\odot
\alpha}_{\eps,\sharp}(t)$ on the space $\overline{\K_{\alpha
C}}$  with
generator $\hat{L}^{\odot
\alpha}_{\eps,\sharp}=\hat{L}_{\eps,\sharp,\mathrm{ren}}^\ast$
 on the domain
 \[
 \Dom\bigl( L^{\odot \alpha}_{\eps,\sharp}\bigr)
 = \bigl\{k\in \overline{\K_{\alpha
C}} \bigm|\hat{L}_{\eps,\sharp,\mathrm{ren}}^\ast k \in
\overline{\K_{\alpha C}}  \bigr\}.
\]
Here and below `\,$\sharp$\/' means `\,{\rm est}\/' or
`\,{\rm
fec}\/', correspondingly. Note that, for $k\in\K_{\alpha
C}$
\begin{eqnarray}\label{oper_adj_eps}
(\hat{L}_{\eps,\sharp, \mathrm{ren}}^\ast
k)(\eta)&=&-m|\eta|k(\eta) \\
&&+\sum_{x\in \eta }\int_{\Ga _{0}}k(\xi \cup (\eta
\setminus
x))\eps^{-|\xi|} \bigl( K_0^{-1}b_{\eps,\sharp} (x,\cdot\cup
\eta\setminus
x)\bigr) (\xi)d\la (\xi).\nonumber
\end{eqnarray}
\end{proposition}

By \eqref{term-est-eps}, \eqref{term-fec-eps}, there exist
the
following point-wise
limits
\begin{eqnarray}
&&\lim_{\eps\rightarrow 0}\eps^{-|\eta|}\left(
K_{0}^{-1}b_{\eps,\mathrm{est}}\left( x,\xi \cup \cdot
\right)
\right) \left( \eta \right) \label{BV-est}\\
&=&\,\varkappa
^{+}\sum_{y'\in \eta }e_{\lambda }\left(-
\phi(x-\cdot),\eta \setminus y'\right) a^{+}(x-y^{\prime
})\nonumber\\
&&+\sum_{y \in \eta
}a^+(x-y)\sum_{y' \in \eta\setminus y
} b^{+}(y-y')
e_{\lambda }\left( -\phi(x-\cdot),\eta \setminus \left\{
y,y'\right\} \right)=:B_x^{V,\mathrm{est}}(\eta)\nonumber
\end{eqnarray}
and
\begin{eqnarray}
&&\lim_{\eps\rightarrow 0}
\eps^{-|\eta|}\left( K_{0}^{-1}b_{\eps,\mathrm{fec}}\left(
x,\xi \cup \cdot \right)
\right) \left( \eta \right)  \label{BV-fec}\\&=& \,\varkappa
^{+}\sum_{y\in \eta }e_{\lambda }\left( -\phi(y-\cdot),\eta
\setminus y\right)  a^{+}(x-y)  \nonumber\\
&&+\sum_{y\in \eta } a^{+}(x-y)\sum_{y'\in \eta
\setminus y} b^{+}(y-y')e_{\lambda
}\left(  -\phi(y-\cdot),\eta \setminus
y\setminus y'\right) =:B_x^{V,\mathrm{fec}}(\eta).\nonumber
\end{eqnarray}
It is worth pointing out that these limits do not depend on $\xi$.
Hence, we have
point-wise limits for $\hat{L}_{\eps, \sharp,
\mathrm{ren}}$:
\begin{equation}
(\hat{L}_{V,\sharp}G)(\eta ):=-m|\eta|G(\eta)+\sum_{\xi
\subset \eta
}\int_{\mathbb{R}^{d}}\,G(\xi \cup
x)B_{x}^{V,\sharp}(\eta\setminus \xi)
dx.\label{newexprV}
\end{equation}

The convergences \eqref{BV-est} and \eqref{BV-fec} in the space $\L_C$ are established
by
our next Lemma.
\begin{lemma}\label{convinLC} Let conditions of
Lemma~\ref{lemma-eps}
hold. Then, for a.a. $x\in\X$ and for $\la$-a.a.
$\xi\in\Ga_0$, the convergence \eqref{BV-est}
and \eqref{BV-fec} hold in the sense of norm
of $\L_C$.
\end{lemma}
\begin{proof}
By \eqref{term-est-eps} and \eqref{term-fec-eps}, it is easy to see that
$\eps^{-|\eta|}\left( K_{0}^{-1}b_{\eps,\sharp}\left( x,\xi
\cup \cdot \right)
\right) \left( \eta \right)$ has the form  $A_\eps(\eta)+\eps
B_\eps(\eta)$. Moreover, the~proof of Lemma~\ref{lemma-eps} assures that $B_\eps$
has an integrable majorant. Hence, by the dominated convergence
theorem, $\eps B_\eps\rightarrow
0$ in $\LC$. Next, using again the dominated convergence theorem
and taking into account
\eqref{BV-est} and \eqref{BV-fec}, we will be able to show convergence of
$A_\eps$ to $B_x^{V,\sharp}$ in $\LC$ once we find uniform in $\eps$
integrable estimate for the corresponding differences
$|A_\eps-B_x^{V,\sharp}|$.

Since $e^{-\eps \phi}\leq1$ and $\psi_\eps(x)<\phi(x)$, for the establishment case,
we have
\begin{eqnarray*}
&&\,\varkappa
^{+}\sum_{y'\in \eta }a^{+}(x-y^{\prime
})\\&&\qquad\times\biggl|e^{-\eps E^{\phi }(x,\xi )}e^{-\eps
\phi (x-y')}e_{\lambda }\left(
\psi_\eps(x-\cdot),\eta \setminus y'\right) -e_{\lambda
}\left(-
\phi(x-\cdot),\eta \setminus y'\right)\biggr|\\
&&+\sum_{y \in \eta
}\sum_{y' \in \eta\setminus y
}a^{+}(x-y) b^{+}(y-y')\\&&\qquad\times\biggl| e^{-\eps E^{\phi
}(x,\xi )}e^{-\eps \phi (x-y) }e^{-\eps \phi (x-y')}
e_{\lambda }\left( \psi_\eps(x-\cdot),\eta \setminus \left\{
y,y'\right\} \right)\\&&\qquad\qquad -
e_{\lambda }\left( -\phi(x-\cdot),\eta \setminus \left\{
y,y'\right\}\right) \biggr|\\ &\leq&\,2\varkappa
^{+}\sum_{y'\in \eta }a^{+}(x-y^{\prime
})e_{\lambda }\left(
\phi(x-\cdot),\eta \setminus y'\right)\\&&+2\sum_{y \in \eta
}\sum_{y' \in \eta\setminus y
}a^{+}(x-y) b^{+}(y-y')e_{\lambda }\left( \phi(x-\cdot),\eta
\setminus \left\{ y,y'\right\}\right).
\end{eqnarray*}
The last expression is an element of $\L_C$, in view of \eqref{minlosid} and
\eqref{LP-exp-mean}. Indeed,
\[
\int_{\Ga_0} \sum_{y'\in \eta }a^{+}(x-y^{\prime
})e_{\lambda }\left(
\phi(x-\cdot),\eta \setminus
y'\right)C^{|\eta|}d\la(\eta)=e^{C\langle\phi\rangle},
\]
and, a similar equality holds for the second term.

One can get the same result for the fecundity case in a similar way.
\end{proof}

Let us denote by $\bar{B}_c^\infty$ the closed ball of radius $c>0$ in
the Banach space $L^\infty(\X)$.

Using Lemma~\ref{convinLC} one can easily pass to the limit in
\eqref{sufcond-eps}. Therefore, in view of the
general results presented in \cite{FKK2011a} we are able to state now the main theorem of this section.

\begin{theorem}[{Proposition 4.2, Theorem
4.4. of \cite{FKK2011a}}]
\label{prop_exist_V}Let the conditions of
Lemma~\ref{lemma-eps}
hold. Then
\begin{enumerate}
\item
$\bigl(\hat{L}_{V,\sharp}, \D\bigr)$ are generators of
a holomorphic semigroups $\hat{U}_{V,\sharp}(t)$ on $\L_C$.
\item $\hat{U}_{\eps,\sharp}(t)\longrightarrow
    \hat{U}_{V,\sharp}(t)$ strongly in
$\L_C$ uniformly on finite time intervals.
\item There exists $\alpha_0\in(0;1)$ such that for any
$\alpha\in(\alpha_0;1)$ the operator $\hat{L}^{\odot
\alpha}_{V,\sharp}=\hat{L}_{V,\sharp}^\ast$
with the domain
 \[
 \Dom\bigl( L^{\odot \alpha}_{V,\sharp}\bigr)
 = \bigl\{k\in \overline{\K_{\alpha
C}} \bigm|\hat{L}_{V,\sharp}^\ast k \in
\overline{\K_{\alpha C}}  \bigr\}.
\]
will be a generator of a strongly continuous
semigroup
$\hat{U}^{\odot \alpha}_{V,\sharp}(t)$ on the space
$\overline{\K_{\alpha
C}}$.
Moreover, for $k\in\K_{\alpha C}$
\begin{equation}\label{oper_adj_V}
(\hat{L}_{V,\sharp}^\ast k)(\eta)=-m|\eta|k(\eta)
+\sum_{x\in \eta }\int_{\Gamma _{0}}k(\xi \cup (\eta
\setminus
x))B_x^{V,\sharp}(\xi)d\lambda (\xi).\nonumber
\end{equation}
\item
Let $\alpha\in(\alpha_0;1)$, $\rho_0\in\bar{B}_{\alpha
C}^\infty$. Then the evolution equation
\[
\left\{
\begin{array}{l}
\frac{\displaystyle \partial }{\displaystyle \partial t} k_t=\hat{L}^\ast_V
k_t\\[2mm]
k_t\bigr|_{t=0}=e_\la(\rho_{0},\eta)
\end{array}\right.
\]
has a unique solution $k_t=e_\la(\rho_t)$ in
$\overline{\K_{\alpha
C}}$ provided $\rho_t$ belongs to $\bar{B}_{\alpha
C}^\infty$ and
satisfies the Vlasov-type equation
\begin{equation} \label{Vlasov_eqn}
\frac{\partial}{\partial t}
\rho_t(x)=-m\rho_t(x)+\int_{\Gamma
_{0}}e_\la(\rho_t,\xi)B_x^{V,\sharp}(\xi)d\lambda (\xi).
\end{equation}
\end{enumerate}
\end{theorem}

Taking into account the explicit expressions
for $B_x^{V,\sharp}$, one can rewrite \eqref{Vlasov_eqn}
in more simple form. Namely, using \eqref{minlosid}, for the
establishment case we obtain
\begin{eqnarray*}
&&\frac{\partial}{\partial t}
\rho_t(x)=-m\rho_t(x)+\int_{\Gamma
_{0}}\int_{\X}e_\la(\rho_t,\eta\cup y)\varkappa
^{+}e_{\lambda }\left(-
\phi(x-\cdot),\eta \right) a^{+}(x-y)dyd\la(\eta)\\
&&+\int_{\Gamma
_{0}}\int_\X \int_\X e_\la(\rho_t,\eta\cup \{ y,y'\})a^+(x-y)
b^{+}(y-y')
e_{\lambda }\left( -\phi(x-\cdot),\eta  \right)dydy'd\lambda
(\eta),
\end{eqnarray*}
and, by \eqref{LP-exp-mean}, we will have
\begin{eqnarray}\label{Vlasov-eqn-est}
\frac{\partial}{\partial t}
\rho_t(x)&=&-m\rho_t(x)+\varkappa^+(\rho_t*a^+)(x)\exp\bigl\{-(\rho_t*\phi)(x)\bigr\}\\
&&+\bigl(\{(\rho_t*b^+)\rho_t\}*a^+\bigr)(x)\exp\bigl\{-(\rho_t*\phi)(x)\bigr\}
.\nonumber
\end{eqnarray}
Here and below $*$ means usual convolutions of functions on $\X$.

Analogously, for the fecundity case, we obtain
\begin{eqnarray}\label{Vlasov-eqn-fec}
\frac{\partial}{\partial t}
\rho_t(x)&=&-m\rho_t(x)+\varkappa^+\bigl(\{\rho_t\exp(-\rho_t*\phi)\}*a^+\bigr)(x)\\
&&+\bigl(\{(\rho_t*b^+)\rho_t\exp(-\rho_t*\phi)\}*a^+\bigr)(x)
.\nonumber
\end{eqnarray}

Of course, we are mostly interesting in nonnegative solution
of Vlasov equation to have $k_t=e_\la(\rho_t)$ is a correlation
function of Poisson non-homogeneous measure with intensity $\rho_t$.
The existence and uniqueness of such solution we establishes by the
following propositions.
\begin{proposition}
Suppose there exists $A>0$ such that $0\leq
\max\{a^+(x),b^+(x)\}\leq A\phi(x)$, $x\in\X$.
Let $c>0$ and
\begin{eqnarray}\label{bigmort-spec-est}
\varkappa^+\Bigl( 1+\frac{A}{e}\langle\phi\rangle
\Bigr)+c\langle
b^+\rangle\Bigl(2+\frac{A}{e}\langle\phi\rangle\Bigr)&<&m,\\
\frac{A}{e}\bigl(\varkappa^++\langle b^+\rangle\bigr)&\leq&
m.\label{bigmort-spec-est2}
\end{eqnarray}
Then the equation \eqref{Vlasov-eqn-est} with initial $0\leq \rho_0\in\bar{B}_c^\infty$
has a non-negative solution $\rho_t$. Moreover, $\rho_t\in\bar{B}_c^\infty$ and it is a
unique solution from $\bar{B}_c^\infty$.
\end{proposition}
\begin{proof}
Let us fix some $T>0$ and consider the Banach space $X_T=C([0;T],L^\infty(\X))$
of all continuous functions on $[0;T]$ with
values in $L^\infty(\X)$; the norm on $X_T$ is given by
$$
\|u\|_T:=\max\limits_{t\in[0;T]}\|u_t\|_{L^\infty(\X)}.
$$
We denote by $X_T^+$ the cone of all
nonnegative
functions
from $X_T$. Denote also by $B_{T,c}^+$ the set
of all functions $u$ from $X_T^+$ with $\|u\|_T\leq
c$.

Let $\Phi$ be a  mapping which
assign to any $v\in X_T$ the solution
$u_t$
of the linear Cauchy problem
\[
\left\{\begin{array}{l}
\frac{\displaystyle \partial}{\displaystyle \partial t} u_t(x) =\,  -
mu_{t}(x)+\varkappa^+(v_t*a^+)(x)\exp\bigl\{-(v_t*\phi)(x)\bigr\}\\[2mm]
\hphantom{\frac{\partial}{\partial t} u_t(x) =\,
}\,+\bigl(\{(v_t*b^+)v_t\}*a^+\bigr)(x)\exp\bigl\{-(v_t*\phi)(x)\bigr\},
\\
u_t \bigr|_{t=0}(x)=\,\rho_0(x),
\end{array}
\right.
\]
for a.a. $x\in\X$. Therefore,
\begin{eqnarray}\label{defPhi}
(\Phi v)_t(x)&=&\, e^{-mt}\rho_0(x)\\&&+\int_0^t
e^{-m(t-s)}\varkappa^+(v_s*a^+)(x)\exp\bigl\{-(v_s*\phi)(x)\bigr\}ds\nonumber\\&&+
\int_0^t e^{-m(t-s)}
\bigl(\{(v_s*b^+)v_s\}*a^+\bigr)(x)\exp\bigl\{-(v_s*\phi)(x)\bigr\}
ds.\nonumber
\end{eqnarray}
It is easy to see that $\Phi v\in X_T$. Indeed, one can estimate
\begin{eqnarray*}
\bigl|(\Phi v)_t(x)\bigr|&\leq& |\rho_0(x)|+\bigl(\varkappa^+
\|v\|_T+\langle b^+\rangle\|v\|_T^2\bigr)\int_0^t
e^{-(t-s)m}ds\\&\leq& c +\frac{\varkappa^+ \|v\|_T+\langle
b^+\rangle\|v\|^2_T}{m} ,
\end{eqnarray*}
where we have used the trivial inequality
\begin{equation}\label{H}
\|f\ast
g\|_{L^\infty(\X)}\leq\|f\|_{L^1(\X)}\|g\|_{L^\infty(\X)}, \quad f\in L^1(\X), \ g\in
L^\infty (\X).
\end{equation}
Clearly, $u_t$ solves
\eqref{Vlasov-eqn-est} if and only if $u$
is a fixed point of the mapping $\Phi:X_T\rightarrow
X_T$.

We have that $v\in X_T^+$ implies  $\Phi v\in
X_T^+$. Next, for any $v, w\in X_T^+$
\begin{eqnarray*}
&&\bigl| (\Phi v)_t(x)-(\Phi w)_t(x) \bigr| \\
&\leq&\,\varkappa^+\int_0^t e^{-m(t-s)}\Bigl|
(v_s*a^+)(x)\exp\bigl\{-(v_s*\phi)(x)\bigr\} \\
&&\qquad\qquad\qquad
- (w_s*a^+)(x)\exp\bigl\{-(w_s*\phi)(x)\bigr\}\Bigr|ds
\\ &&+\int_0^t e^{-m(t-s)}\Bigl|
\bigl(\{(v_s*b^+)v_s\}*a^+\bigr)(x)
\exp\bigl\{-(v_s*\phi)(x)\bigr\}\\
&&\qquad\qquad\qquad-\bigl(\{(w_s*b^+)w_s\}*a^+\bigr)(x)
\exp\bigl\{-(w_s*\phi)(x)\bigr\}\Bigr|ds.
\end{eqnarray*}
Taking into account \eqref{H} and obvious inequalities
$e^{-x}x\leq
e^{-1}$ for $x\geq0$, \mbox{$|e^{-a}-e^{-b}|\leq |a-b|$} for
$a,b\geq0$, and, moreover,
\[
|pe^{-a}-qe^{-b}|\leq e^{-a}|p-q|+qe^{-b}|e^{-(a-b)}-1|\leq
e^{-a}|p-q|+qe^{-b}|a-b|,
\]
for any $a,b,p,q\geq0$,
we obtain
\begin{eqnarray*}
&&\varkappa^+\int_0^t e^{-m(t-s)}\Bigl|
(v_s*a^+)(x)\exp\bigl\{-(v_s*\phi)(x)\bigr\} \\
&&\qquad\qquad\qquad -
(w_s*a^+)(x)\exp\bigl\{-(w_s*\phi)(x)\bigr\}\Bigr|ds\\
&\leq&\,\varkappa^+\int_0^t e^{-m(t-s)}\Bigl(
(|v_s-w_s|*a^+)(x)\exp\bigl\{-(v_s*\phi)(x)\bigr\} \\ && +
(w_s*a^+)(x)\exp\bigl\{-(w_s*\phi)(x)\bigr\}
(|v_s-w_s|*\phi)(x)\Bigr)ds\\
&\leq & \, \varkappa^+\|v-w\|_T\Bigl(
1+\frac{A}{e}\langle\phi\rangle \Bigr)\int_0^t
e^{-m(t-s)}ds\leq\|v-w\|_T\frac{\varkappa^+}{m}\Bigl(
1+\frac{A}{e}\langle\phi\rangle \Bigr);
\end{eqnarray*}
and, similarly,
\begin{eqnarray}
&&\int_0^t e^{-m(t-s)}\Bigl|
\bigl(\{(v_s*b^+)v_s\}*a^+\bigr)(x)\exp\bigl\{-(v_s*\phi)(x)\bigr\}\nonumber\\
&&-\bigl(\{(w_s*b^+)w_s\}*a^+\bigr)(x)\exp\bigl\{-(w_s*\phi)(x)\bigr\}\Bigr|ds\nonumber\\
&\leq&\, \int_0^t e^{-m(t-s)}\Bigl(
\bigl(\bigl|(v_s*b^+)v_s-(w_s*b^+)w_s\bigr|*a^+\bigr)(x)\exp\bigl\{-(v_s*\phi)(x)\bigr\}\nonumber
\\&&+\bigl(\{(w_s*b^+)w_s\}*a^+\bigr)(x)\exp\bigl\{-(w_s*\phi)(x)\bigr\}
(|v_s-w_s|*\phi)(x)\Bigr)ds.\label{dop1}
\end{eqnarray}
Using the bound
\begin{eqnarray*}
&&\bigl(\{(w_s*b^+)w_s\}*a^+\bigr)(x)\exp\bigl\{-(w_s*\phi)(x)\bigr\}\\
&\leq&\,
\|w_s*b^+\|_{L^\infty(\X)}(w_s*a^+)(x)\exp\bigl\{-(w_s*\phi)(x)\bigr\},
\end{eqnarray*}
we may continue to estimate \eqref{dop1} as follows
\begin{eqnarray*}
&\leq&\int_0^t
e^{-m(t-s)}\Bigl\|(v_s*b^+)v_s-(v_s*b^+)w_s\Bigr\|_{L^\infty(\X)}ds\\
&&+\int_0^t
e^{-m(t-s)}\Bigl\|(v_s*b^+)w_s-(w_s*b^+)w_s\Bigr\|_{L^\infty(\X)}ds
\\&&+\int_0^t
e^{-m(t-s)}\bigl\|w_s\bigr\|_{L^\infty(\X)}\langle
b^+\rangle\frac{A}{e}\|v-w\|_T\langle\phi\rangle
ds.
\end{eqnarray*}
For $\|v\|_t\leq c$, $\|w\|_T\leq c$ one can estimate this expression by
\[
\Bigl(2c \|v-w\|_T\langle b^+\rangle+c\langle
b^+\rangle\frac{A}{e}\|v-w\|_T\langle\phi\rangle\Bigr)
\int_0^t e^{-m(t-s)}ds.
\]
Therefore, for $v,w\in X_T^+$, $\|v\|_T\leq
c$, $\|w\|_T\leq c$
\[
\|\Phi v-\Phi w\|_T\leq \frac{\varkappa^+}{m}\Bigl(
1+\frac{A}{e}\langle\phi\rangle \Bigr)\|v-w\|_T+\frac{c\langle
b^+\rangle}{m}\Bigl(2+\frac{A}{e}\langle\phi\rangle\Bigr)\|v-w\|_T.
\]

Moreover, if $\rho_0\in\bar{B}_c^\infty$
and $v\in B_{T,c}^+$ then, by \eqref{defPhi},
\begin{eqnarray*}
|(\Phi v)_t(x)|&\leq& e^{-mt}c
+\frac{A\varkappa^+}{me}\bigl(1-e^{-mt}\bigr)c+
c\langle b^+\rangle\frac{A}{me}\bigl(1-e^{-mt}\bigr)\\
& =&\frac{cA}{me}\bigl(\varkappa^++\langle
b^+\rangle\bigr)+e^{-mt}c\Bigl(1-\frac{A}{me}\bigl(\varkappa^++\langle
b^+\rangle\bigr)\Bigr)\leq
c,
\end{eqnarray*}
provided \eqref{bigmort-spec-est2} holds.

As a result, by
\eqref{bigmort-spec-est}, \eqref{bigmort-spec-est2}, $\Phi$
is a contraction mapping on the closed set $B_{T,c}^+$.
Taking, as usual,
$v^{(n)}=\Phi^nv^{(0)}$, $n\geq1$ for $v^{(0)}\in B_{T,c}^+$
we obtain
that $\{v^{(n)}\}\subset B_{T,c}^+$ is a fundamental sequence
in $X_T$
which has, as a result, a unique limit point $v\in X_T$. Since
$B_{T,c}^+$
is a closed set we have that $v\in B_{T,c}^+$. Then,
according to the
classical Banach fixed point theorem, $v$ will be a fixed
point of
$\Phi$ on $X_T$ and a unique fixed point on $B_{T,c}^+$.
\end{proof}

The same considerations may be applied to the Vlasov equation
\eqref{Vlasov-eqn-fec}. To combine these results with
statement of
Theorem~\ref{prop_exist_V} we need additionally that
\eqref{bigmort-spec-est}, \eqref{bigmort-spec-est2} hold with
$c=\alpha C$.

\section*{Acknowledgments}
The financial support of DFG through the SFB 701
(Bielefeld University) and German-Ukrainian Project KO 1989/6-1 is gratefully
acknowledged.

\end{document}